\documentclass[12pt]{article}

\usepackage[usenames]{xcolor}
\usepackage{graphicx}
\usepackage[colorlinks=true,
linkcolor=webgreen,
filecolor=webbrown,
citecolor=webgreen]{hyperref}
\definecolor{webgreen}{rgb}{0,.5,0}
\definecolor{webbrown}{rgb}{.6,0,0}
\usepackage{color}

\usepackage{amssymb}
\usepackage{amsthm}
\usepackage{mathtools}
\usepackage{algpseudocode,algorithm,algorithmicx}
\usepackage{bm}

\usepackage{tikz}
\usepackage{subcaption}
\usepackage{verbatim}
\usepackage{diagbox}
\usepackage{float}

\usepackage{authblk}
\usepackage{setspace}

\theoremstyle{plain}
\newtheorem{theorem}{Theorem}
\newtheorem{corollary}[theorem]{Corollary}
\newtheorem{lemma}[theorem]{Lemma}
\newtheorem{proposition}[theorem]{Proposition}

\theoremstyle{definition}
\newtheorem{definition}[theorem]{Definition}

\theoremstyle{remark}

\newcommand{\floor}[1]{\left\lfloor#1\right\rfloor}
\newcommand{\ceil}[1]{\left\lceil#1\right\rceil}

\makeatletter
\renewcommand{\Function}[2]{%
	\csname ALG@cmd@\ALG@L @Function\endcsname{#1}{#2}%
	\def\curfunc{#1}%
}
\newcommand{\funclabel}[1]{%
	\@bsphack
	\protected@write\@auxout{}{%
		\string\newlabel{#1}{{\curfunc}{\thepage}}%
	}%
	\@esphack
}

\begin{document}
	\author{Laura Monroe}
	\affil{Ultrascale Systems Research Center, Los Alamos National Laboratory, Los Alamos, NM 87501
		\newline
		lmonroe@lanl.gov}
	\renewcommand\Affilfont{\itshape\small}
	\renewcommand\footnotemark{}
	\thanks{		
		This publication has been assigned the LANL identifier LA-UR-21-25560. 
		
		\hspace*{0.82em}This work has been authored by an employee of Triad National Security, LLC, operator of the Los Alamos National Laboratory under Contract No.89233218CNA000001 with the U.S. Department of Energy. This work was also supported by LANL's Ultrascale Systems Research Center at the New Mexico Consortium (Contract No. DE-FC02-06ER25750).).
		The United States Government retains and the publisher, by accepting this work for publication, acknowledges that the United States Government retains a nonexclusive, paid-up, irrevocable, world-wide license to publish or reproduce this work, or allow others to do so for United States Government purposes.  
	} 
	\date{\vspace{-5ex}}
	\title{Binary Signed-Digit Integers and the Stern Polynomial
	}
	
	\maketitle

\begin{abstract}
	The binary signed-digit representation of integers is used for efficient computation in various settings. The Stern polynomial is a polynomial extension of the well-studied Stern diatomic sequence, and has itself has been investigated in some depth. 
	In this paper, we show previously unknown connections between BSD representations and the Stern polynomial.
	
	We derive a weight-distribution theorem for $i$-bit BSD representations of an integer $n$ in terms of the coefficients and degrees of the terms of the Stern polynomial of $2^i-n$. 
	
	We then show new recursions on Stern polynomials, and from these and the weight-distribution theorem obtain similar BSD recursions and a fast $\mathcal{O}(n)$ algorithm that calculates the number and number of $0$s of the optimal BSD representations of all of the integers of NAF-bitlength $\log(n)$ at once, which then may be compared.
	\newline\newline	
\textbf{Keywords:} Binary signed-digit representations; hyperbinary representations; non-adjacent form, Stern polynomial.
\newline\newline
\textbf{Mathematics Subject Classification (2010):} 11A63 $\cdot$ 11B83 $\cdot$ 68R01 
\end{abstract}

\section{Introduction}
Integers may be represented in binary signed-digit (BSD) representation, in which each integer is represented in terms of sums or differences of distinct powers of $2$ \cite{shannon50,booth1951,Shallit92aprimer}. 
BSD representations of an integer are not unique; there are an infinite number of such representations for any integer using an arbitrary number of bits, and of course, a finite number for a fixed number of bits. 

The number of such representations of an integer is of interest, as is the number of representations having a maximal number of $0$s, as these optimal representations may be used for efficient calculation~\cite{shannon50,booth1951,avizienis1961signed}, and in many other applications \cite{morain90,egecioglu90,koblitz91}. 
At times, it can be useful to have a choice of 
integers of a given length, each having many optimal representations, or  optimal representations with many $0$s, or both. 

In \cite{monroe21}, a correspondence was shown between the number of BSD representations of an integer and Stern's diatomic sequence. We extend that work here, and examine the relationship of the related Stern polynomial \cite{KLAVZAR200786} to the number of integer BSD representations, and in particular to optimal representations. 

The first result of this paper is a weight-distribution theorem identifying the number of $i$-bit BSD representations of an integer $n$ having weight $i-\ell$ with the coefficient of the $\ell^\text{th}$ term of the Stern polynomial of $2^i -n$. 
This refines the result in \cite{monroe21}, and is the basis for the rest of the paper. 

As a consequence of this theorem, the number of $i$-bit optimal BSD representations of an integer $n$ is the same as the leading coefficient of the Stern polynomial $B_{2^i-n}$, and the number of $0$s in an optimal representation of $n$ is the degree of $B_{2^i-n}$. 

In an effort to express these Stern parameters more directly, we derive a simple recursion on the leading coefficients and degrees of Stern polynomials of integers. 
This recursion depends on a partition of $I_k$, the interval of integers having non-adjacent form (NAF) of bitlength $k$. 
$I_k$ is partitioned into three subintervals $\mathcal{A}_k$, $\mathcal{B}_k$ and $\mathcal{C}_k$,
as shown in Fig.~\ref{fig:partition_Ik}. 
The leading coefficients and degrees of the Stern polynomials of integers in $\mathcal{A}_k$ and $\mathcal{C}_k$ are expressed in terms of those of $I_{k-2}$ and those of $\mathcal{B}_k$ are expressed in terms of those of $I_{k-1}$, in forward and reverse order. This is the second result of the paper.

The weight-distribution theorem is then applied to the recursion on Stern polynomials, and similar recursive formulas based on the partition are established for the number of optimal BSD representations of integers and for the weights of these optimal representations. This is the third result of the paper.

Finally, the simplicity of the recursions  is illustrated with algorithms that calculate the number and number of $0$s of the optimal BSD representations of all integers in $I_k$. These algorithms are $\mathcal{O}(n)$, where $n$ is the maximum integer calculated. Their advantage is that they permit comparison between all integers of a given NAF-length that have weights and number of optimal representations appropriate to the application in question.
\section{The Stern polynomial}\label{section_sternpoly}
 The Stern polynomial of $n$, $B_n(t)$, was introduced by Klavžar et al. in \cite{KLAVZAR200786}. This polynomial is closely related to the Stern diatomic sequence. Among other things, Klavžar et al. show that $B_n(1) = c(n)$, where $c(n)$ is the $n^{\text{th}}$ entry in the Stern sequence, and $B_n(2) = n$.
 \begin{definition}[Stern polynomial]~\cite{KLAVZAR200786}\label{def_stern_poly}
    Let $n \in \mathbb{N}_0$. The \emph{Stern polynomial} $B_n(t)$ is defined as follows:
    	$$
		B_n(t)= 
		\begin{cases}
		0 & \text{if $n=0$,} \\
		1 & \text{if $n=1$,} \\
		t\cdot B_m(t) & \text{if $n=2m$,} \\
		B_m(t)+B_{m+1}(t) & \text{if $n=2m+1$.} 
		\end{cases}
	$$ 
For ease of notation, throughout the rest of this paper we will refer to $B_n(t)$ simply as $B_n$. We refer to the leading coefficient of $B_n$ as $\ell c(B_n)$, and to the degree of $B_n$ as $deg(B_n)$.
\end{definition}
Many papers have addressed aspects of the Stern polynomial. Ulas and Ulas discuss some of its arithmetic properties in \cite{ulas2011} and Ulas continues this investigation in \cite{ulas2012}. Schinzel investigates the factors of Stern polynomials in \cite{Schinzel2011}, and presents a formula for the leading coefficient of $B_n(t)$ in terms of the binary representation of $n$ in \cite{Schinzel2016}. Dilcher and Tomkins continue the investigation into divisibility of Stern polynomials in \cite{dilcher10}.
  
 The Stern polynomial and its associated properties are extremely useful in the study of BSD representations of integers. 
 We discuss these applications throughout the rest of this paper.
\section{The weight distribution of the BSD representations of an integer}\label{sec_weights}
As our first result, we derive a weight distribution of $i$-bit BSD representations of $n$, in terms of the Stern polynomials $B_{2^i-n}(t)$. This permits the application of the many prior results on the Stern polynomial to the study of BSD representations of integers.

The result given here is exact, so improves upon the upper bound given in \cite{Tuma_2015} by T{\r{u}}ma and V\'abek for the number of $i$-bit BSD representations of an integer $n$ of a given weight. 

This derivation depends on a transform  between the $i$-bit BSD representations of $n$ and the $i$-bit hyperbinary representations of $(2^i-1-n)$. 
\begin{definition}[BSD representation of an integer]\label{def_bsd}
	An integer $n$ is in \emph{BSD representation} when
	$$
		n=\sum_{j=0}^{i-1}b_j2^j \text{, where $b_j \in \{1,0,-1\}$.}
    $$
\end{definition}
 \begin{definition}[Hyperbinary representation of an integer]~\cite{reznick90}
	An integer $n$ is in \emph{hyperbinary representation} when
	$$
		n=\sum_{j=0}^{i-1}h_j2^j \text{, where $h_j \in \{0,1,2\}$.}
	$$
\end{definition}
In \cite{monroe21}, Monroe noted a correspondence between the BSD representations of $n$ and the hyperbinary representations of $2^i-1-n$. This correspondence is restated here in Theorem~\ref{theorem_bsd_hb_rep}. \begin{theorem}~\cite{monroe21}\label{theorem_bsd_hb_rep}
    $(b_{i-1}\cdots b_0)$ is an $i$-bit BSD representation of a non-negative integer $n$ if and only if $((1-b_{i-1})\cdots (1-b_0))$ is an $i$-bit hyperbinary representation of $2^i-1-n$.
\end{theorem}
As noted in \cite{monroe21}, Theorem~\ref{theorem_bsd_hb_rep} gives the following transform between the digits of a BSD representation $(b_{i-1}\cdots b_0)$ of an integer $n$ and the digits of the corresponding hyperbinary representation, $(h_{i-1}\cdots h_0)$ of $2^i-1-n$.
\begin{equation}\label{eq_translation}
\begin{aligned}
   1\leftrightarrow 0&,\\ 0\leftrightarrow 1&,\\
  -1\leftrightarrow 2&. 
\end{aligned}
\end{equation}
Klavžar et al. \cite{KLAVZAR200786} relate the coefficients of Stern polynomials to hyperbinary representations in the following theorem.
\begin{theorem}~\cite{KLAVZAR200786}\label{theorem_ones_in_hb}
    Let 
    $h_1(n,\ell)$
    be the number of hyperbinary representations of $n \in \mathbb{N}$ that have exactly $\ell$ $1$s. Then
    $$
    B_n=\sum_{\ell \ge 0}
    h_1(n-1,\ell)\cdot
    t^\ell.
    $$
\end{theorem}
Using Klavžar's Theorem~\ref{theorem_ones_in_hb} and the transform in Equation (\ref{eq_translation}), we obtain a weight distribution of the $i$-bit BSD representations of an integer $n$ in terms of the Stern polynomial of $2^i-n$.
\begin{theorem}~\label{theorem_weightdist}
    Let $b_0(n,i,\ell)$ be the number of $i$-bit BSD representations of $n \in \mathbb{N}$ that have exactly $\ell$ $0$s (and thus have weight $i-\ell$). Then
    $$
    B_{2^i-n}=\sum_{\ell \ge 0}
    b_0(n,i,\ell)\cdot
    t^\ell.
    $$
\end{theorem}
\begin{proof}
    By Theorem~\ref{theorem_ones_in_hb},
    $$B_{2^i-n}=\sum_{\ell \ge 0}
    h_1(2^i-n-1,\ell)\cdot
    t^\ell.$$ By Equation (\ref{eq_translation}), there is a one-to-one correspondence between the $i$-bit hyperbinary representations of $2^i-1-n$ having exactly $\ell$ $1$s and the $i$-bit BSD representations of $n$ having exactly $\ell$ $0$s. 
    So $h_1(2^i-n-1,\ell)=b_0(n,i, \ell).$ 
\end{proof}
 \begin{corollary}\label{cor_naf}
    Let $n \in I_k$. Then the number of $0$s in the reduced NAF of $n$ is $deg(B_{2^k-n})$, and its weight is $k-deg(B_{2^k-n})$.
\end{corollary}
\begin{proof}
    This follows immediately from the weight distribution result in Theorem~\ref{theorem_weightdist}, and the fact that the NAF of $n$ has minimum weight.
\end{proof}
\section{Non-adjacent form and the NAF-interval $I_k$ }
Reitweisner introduced the non-adjacent form (NAF) of an integer in \cite{REITWIESNER1960231}. The NAF-interval $I_k$ turns out to be very useful in the derivation of several properties of Stern polynomials, and therefore of BSD representations. 

Using the definitions in this section, we will calculate degrees and leading coefficients of Stern polynomials recursively across NAF-intervals. The weight-distribution Theorem~\ref{theorem_weightdist} then allows us to extend that result to count the number of optimal BSD representations of integers along with their weights.

\begin{definition}[Non-adjacent form]~\cite{REITWIESNER1960231}
    The $k$-bit \emph{non-adjacent form} (NAF) of an integer $n$ is a BSD representation $(b_{k-1}\cdots b_0)$ of $n$, with the property that for all $0 \le i<k-1$, either $b_i$ or $b_{i+1}$ must be $0$. An NAF is said to be \emph{reduced} if $b_{k-1} \ne 0$.
\end{definition}
Every integer $n$ has exactly one reduced NAF. The NAF of $n$ has the most $0$s of any BSD representation of $n$ of the same length, so the smallest weight, thus making the NAF good for fast arithmetic. However, the NAF of $n$ may not be the only BSD representation of its length having minimal weight. 

\begin{definition}[NAF-bitlength]
    The \emph{NAF-bitlength} of $n$ is the bitlength of the reduced NAF of $n$.
\end{definition}
\begin{definition}[NAF-interval]
    The \emph{NAF-interval of bitlength $k$} is the interval of positive integers having NAF-bitlength $k$. We denote this as $I_k$.
\end{definition}
The following simple lemma shows a symmetry on the NAF-interval $I_k$: all integers in $I_k$ come in pairs $(n, 2^k-n)$, except for $2^{k-1}$, which is the midpoint of $I_k$ and is paired with itself. 
\begin{lemma}\label{lemma_naf_symmetricity}
    An integer $n \in I_k$ if and only if $2^k-n \in I_k$.  
\end{lemma}
\begin{proof}
    Let $n$ be positive and let $\bar{n}_i = -n_i$. The NAF of $n$ is $(1\ \ 0 \ \  n_{k-3}\dots n_1\ \ n_0)$ if and only if the NAF of $2^k-n$ is $(1 \ \ 0 \ \ \bar{n}_{k-3}\dots \bar{n}_1\ \ \bar{n}_0)$. Both of these have NAF-bitlength $k$. 
\end{proof}
\begin{definition}[Sibling integers]
    Let $n  \in I_k$ . The \emph{sibling} of $n$ is $2^k-n$. 
\end{definition}
Siblings appear throughout this paper, and have already appeared in Theorem~\ref{theorem_weightdist}, giving the weight distribution of the BSD representations of $n$ in terms of the Stern polynomial of its sibling $2^k-n$.

\section{A partition on the NAF-interval $I_k$}
The following partition of $I_k$ is the basis for the recursions shown in subsequent sections. We will draw correspondences between $\mathcal{A}_k$, $\mathcal{C}_k$ and $I_{k-2}$, and between $\mathcal{B}_k$ and the first half of $I_{k-1}$, and use these to establish recursions for Stern polynomials and BSDs of integers in $I_k$. The partition is illustrated in Fig.~\ref{fig:partition_Ik}.

The interval is partitioned into three disjoint subintervals, so $I_k = \mathcal{A}_k \cup \mathcal{B}_k \cup \mathcal{C}_k$. The integers $a_k$, $b_k$, and $c_k$ are the lower bounds of the intervals $\mathcal{A}_k$, $\mathcal{B}_k$, and $\mathcal{C}_k$. The first and third of these subintervals, $\mathcal{A}_k$ and $\mathcal{C}_k$, have the same length as $I_{k-2}$. The middle subinterval, $\mathcal{B}_k$, has the same length as $I_{k-1}$.
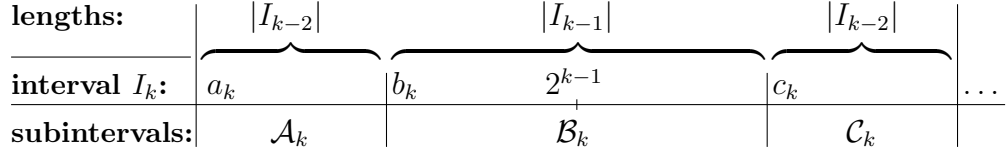
\begin{figure}[h!t]
\centering
\setlength\unitlength{2pt}
\begin{picture}(168,35)

    \put(-10,9){\line(1,0){188}}

    \put(-10,24){\small \bf{lengths:}}
    \put(-10,18){\line(1,0){34}}
    \put(-10,11){\small \bf{interval $I_k$:}}
    \put(-10,2){\small \bf{subintervals:}}
    
    \put(25,1){\line(0,1){27}}
    \put(26,17){$\overbrace{\hspace{5.75em}}$}
    \put(35,24){$\lvert I_{k-2}\rvert$}
    \put(27,11){$a_k$}
    \put(39,2){$\mathcal{A}_k$}

    \put(61,1){\line(0,1){14}}
    \put(62,17){$\overbrace{\hspace{12em}}$}
    \put(90,24){$\lvert I_{k-1}\rvert$}
    \put(62,11){$b_k$}

    \put(97,8){\line(0,1){2}}
    \put(91,11){$2^{k-1}$}
    \put(93,2){$\mathcal{B}_k$}

    \put(133,1){\line(0,1){14}}
    \put(134,17){$\overbrace{\hspace{5.75em}}$}
    \put(144,24){$\lvert I_{k-2}\rvert$}
    \put(134,11){$c_k$}
    \put(148,2){$\mathcal{C}_k$}

    \put(169,1){\line(0,1){27}}
    \put(170,11){$\dots$}

\end{picture}
    \caption{Partition of $I_k$, the interval of integers having NAF-bitlength $k \ge 3$. The values above the brackets indicate the length of the subintervals. The integers $a_k$, $b_k$, and $c_k$ are the lower bounds of the intervals $\mathcal{A}_k$, $\mathcal{B}_k$, and $\mathcal{C}_k$, respectively. $2^{k-1}$ is the midpoint of $I_k$.}
    \label{fig:partition_Ik}
\end{figure}
\begin{lemma}\label{lemma_interval_nafX}
    The interval of integers $I_k$ is $$[a_k, a_{k+1}) = \left[ \ceil{\frac{2^k}{3}}, \ceil{\frac{2^{k+1}}{3}}\right).$$
\end{lemma}
\begin{proof}
    This proof rests on the fact that $2^k=1 \bmod 3$, if $k$ is even, and $2^k=2 \bmod 3$, if $k$ is odd. Let $n_k$ be the the greatest integer having NAF-bitlength $k$. Then $n_k=(1010 \dots 10)$ if $k$ is even, and $n_k=(1010 \dots 01)$ if $k$ is odd. In either case, $n_k=\floor{\frac{2^{k+1}}{3}}$.
So $a_{k+1}=n_k+1= \ceil{\frac{2^{k+1}}{3}}$, and $a_k= n_{k-1}+1 =\ceil{\frac{2^{k}}{3}}$. 
\end{proof}
The remainder of this section is comprised of a few lemmas on 
the length of $I_k$ and on 
the values of the lower endpoints $a_k, b_k, c_k$ of the partitions, to be used in the recursions in Section~\ref{section_deg_lc_stern} and ~\ref{sec_weights_BSDs}. 
These lemmas follow from Lemma~\ref{lemma_interval_nafX}, the partition definition and some arithmetic.
We do not include the routine proofs of Lemmas~\ref{lemma_length_naf} through \ref{lemma_values_bc} here, in the interest of brevity of exposition.
\begin{lemma}\label{lemma_length_naf} The length of $I_k$, with $k \ge 1$, is 
    $$
\lvert I_{k} \rvert = \lvert I_{k-1} \rvert + 2 \cdot \lvert I_{k-2} \rvert =
    \begin{cases}
        \floor{\frac{2^k}{3}}=a_k-1, &\text{if $k$ is even,}\\[.5em]%
        \ceil{\frac{2^k}{3}}=a_k, &\text{if $k$ is odd.}
    \end{cases}
$$
\end{lemma}
\begin{lemma}\label{lemma_values_ac}
Let $I_k$ be partitioned as in Fig.~\ref{fig:partition_Ik}, where $k \ge 3$. Then
    \begin{align*}
        a_k &= 2^{k-2}+a_{k-2},\\
        c_k &= 2^{k-1}+a_{k-2}.
    \end{align*}
\end{lemma}
\begin{lemma}\label{lemma_values_bc}
Let $I_k$ be partitioned as in Fig.~\ref{fig:partition_Ik}, where $k \ge 3$. 
Then 
    \begin{align*}
        b_k &=
        \begin{cases}
            2^{k-1}-\lvert I_{k-2}\rvert &\text{if $k$ is even,}\\
            2^{k-1}-(\lvert I_{k-2}\rvert-1) &\text{if $k$ is odd.}
        \end{cases}\\
        c_k &=
        \begin{cases}
            2^{k-1}+(\lvert I_{k-2}\rvert+1) &\text{if $k$ is even,}\\
            2^{k-1}+\lvert I_{k-2}\rvert &\text{if $k$ is odd.}
        \end{cases}
    \end{align*}
\end{lemma}
The following lemma gives a relationship between siblings in $I_{k-2}$ and siblings in $I_k$ from $\mathcal{A}_k$ and $\mathcal{C}_k$.
\begin{lemma} \label{lemma_sibs_across_k}
    Let $0 \le v < \lvert I_{k-2}\rvert$. Then
  $$2^{k-2}-(a_{k-2}+v)=a_{k-2}+x \text{ if and only if } 
        2^{k}-(a_k+v)=c_k+x.
    $$
    In other words, $a_{k-2}+v$ is the sibling of $a_{k-2}+x$ in $I_{k-2}$ if and only if $a_k+v$ is the sibling of $c_k+x$ in $I_k$.
\end{lemma}
\begin{proof}
Let $2^{k-2}-(a_{k-2}+v)=a_{k-2}+x$. Then
       \begin{align*}
         2^k-(a_k+v)) 
         &= 2^k-(2^{k-2}+a_{k-2}+v)&\text{by Lemma~\ref{lemma_values_ac}} \\
        &= 2^{k-1}+(2^{k-2}-(a_{k-2}+v))  \\
        &= 2^{k-1}+(a_{k-2}+x))   &\text{by the assumption}\\
        &= c_k+x. &\text{by Lemma~\ref{lemma_values_ac}}
    \end{align*}
Conversely, let $2^{k}-(a_k+v)=c_k+x$. 
This implies $2^{k-1}-(a_k+v)=c_k-2^{k-1}+x$. So
\begin{align*}
    2^{k-2}-(a_{k-2}+v)
    &= 2^{k-1}-2^{k-2}-a_{k-2}-v \\
    &= 2^{k-1}-a_k-v &\text{by Lemma~\ref{lemma_values_ac}}\\
    &= c_k-2^{k-1}+x &\text{by the assumption}\\
    &= a_{k-2}+x. &\text{by Lemma~\ref{lemma_values_ac}}
\end{align*}
\end{proof}
\section{Degrees and leading coefficients of Stern polynomials on $I_k$}\label{section_deg_lc_stern}
In this section, we give a recursion expressing the degrees and leading coefficients of Stern polynomials of integers in $I_k$ in terms of the degrees and leading coefficients of Stern polynomials of integers in $I_{k-2}$ and $I_{k-1}$. 

The idea in this section is to express the Stern polynomial of $n \in I_k$ as a sum involving Stern polynomials of $2^a-r$ and $r$ so that $2^a-r$ (and perhaps $r$) are in $I_{k-1}$ or $I_{k-2}$. We then compare the degrees of these polynomials, which allows us to deduce the leading coefficient and degree of the sum.

We make use of a lemma by Schinzel \cite{Schinzel2011}, extended by Dilcher and Tomkins in \cite{dilcher10}, to obtain the needed expression.
\begin{lemma}\cite{Schinzel2011,dilcher10}\label{lemma_schinzel}
    Let $a, m,\text{ and } r$ be integers such that $0 \le r \le 2^a$. Then
    \begin{align*}
        B_{2^a m+r}&=B_{2^a -r}B_{m}+B_{r}B_{m+1} \text{, and}\\
        B_{2^a m-r}&=B_{2^a -r}B_{m}+B_{r}B_{m-1}.
    \end{align*}
\end{lemma}
By applying Lemma~\ref{lemma_schinzel} with $m=1$ or $m=2$ (which forces the values of $a$ and $r$), a Stern polynomial may be expressed as a combination of the Stern polynomials $B_{2^a -r}$ and $B_{r}$, each multiplied by $1$, $t$, or $(t+1)$. 
Lemma~\ref{lemma_compare_degs} then allows us to compare the degrees of $B_{2^a-r}$ and $B_{r}$, and from that, obtain the degree and leading coefficient of the sums derived from Lemma~\ref{lemma_schinzel}. 
\begin{lemma}\label{lemma_compare_degs}
    Let $a\ge 1$, and let $0 \le s < a$. Let $r \in I_{a-s}$. Then 
    $$
        deg(B_{2^a-r}) = s+deg(B_{r}).
    $$
\end{lemma}
\begin{proof}
     $2^a$ has NAF-bitlength $a+1$ and $r$ has NAF-bitlength $a-s$.\vspace{.5em}
     \newline If $s=0$, $2^a-r$ may be expressed as 
         $$\begin{array}{lllllll}
              & 1 & 0 & 0 &  0 & \cdots &  0\\
            - &  & 1 & 0 &  r_{a-s-2} & \cdots &  r_0\\
            \hline
            \rule{0pt}{2.6ex}
              &  & 1 & 0 &  \overline{r_{a-s-2}} & \cdots & \overline{r_0}
         \end{array}$$
        with $r$ in non-adjacent form, and with $\overline{x}=-x$. By Theorem~\ref{theorem_weightdist}, the number of $0$s in $r$ is $deg(B_r)$. The result is in reduced non-adjacent form, and has the same number of $0$s as $r$. So again by Theorem~\ref{theorem_weightdist}, $deg(B_{2^a-r})=deg(B_r)$.\vspace{.5em}
        \newline If $s>0$, $2^a-r$ may be expressed as 
         $$\begin{array}{llllllllll}
              & 1 & 0 & \cdots & 0 &  0 & 0 &  0 & \cdots &  0\\
              - &  &  &   &  &  1 & 0 &  r_{a-s-2} & \cdots &  r_0\\
              \hline
             \rule{0pt}{2.6ex}
             & 1 & 0 & \cdots & 0 & \overline{1} & 0 & \overline{r_{a-s-2}} & \cdots & \overline{r_0}
         \end{array}$$
        with $r$ in non-adjacent form. By Theorem~\ref{theorem_weightdist}, the number of $0$s in $r$ is $deg(r)$. The result is in reduced non-adjacent form, and has $(a+1)-(a-s)-1 = s$ more $0$s than $r$. So again by Theorem~\ref{theorem_weightdist}, $deg(B_{2^a-r})=s+deg(B_r)$.
\end{proof}
Lemma~\ref{lemma_compare_degs} has as a corollary the symmetry of degrees of Stern polynomials in $I_k$. In other words, siblings in $I_k$ have Stern polynomials of equal degree.
\begin{corollary}\label{cor_naf_deg_symmetry}
    Let $k \ge 1$ and let $n \in I_k$. Then 
    $$
        deg(B_{2^k-n})=deg(B_n).
    $$
\end{corollary}
Lemma~\ref{lemma_compare_degs} requires consideration of integers in terms of their NAF-bitlengths. We have not found a similar relationship in terms of binary bitlengths. The partition $I_k = \mathcal{A}_k \cup \mathcal{B}_k \cup \mathcal{C}_k$ comes into play in the recursion, since $2^a -r$ and $r$ fall into different prior NAF-intervals, for $n$ in $\mathcal{A}_k$, $\mathcal{B}_k$ and $\mathcal{C}_k$.

\subsection{Degrees and leading coefficents of the Stern polynomial in $\mathcal{A}_k$ and $\mathcal{C}_k$}
The Stern polynomials for the integers in both $\mathcal{A}_k$ and $\mathcal{C}_k$ may be expressed as a combination of Stern polynomials of siblings in $I_{k-2}$. This follows by application of Schinzel's Lemma~\ref{lemma_schinzel} and the symmetricity of degrees of siblings.
\begin{proposition}\label{prop_degree_lc_A}
     Let $k \ge 3$, and let $0 \le v < \lvert I_{k-2}\rvert$. Then 
     \begin{align*}
        deg(B_{a_k+v}) &= deg(B_{a_{k-2}+v})+1 \text{, and } \\
        \ell c(B_{a_k+v}) &= \ell c(B_{a_{k-2}+v}).
     \end{align*}
\end{proposition}
\begin{proof}
Setting $a=2^{k-2},\ \ m=1,\text{ and } r=a_{k-2}+v$, we apply Lemma \ref{lemma_schinzel}:
    \begin{align*}
        B_{a_k+v} 
        &= B_{(2^{k-2}+a_{k-2})+v} &\text{by Lemma \ref{lemma_values_ac}}\\
        &= B_{2^{k-2}-(a_{k-2}+v)}+t\cdot B_{a_{k-2}+v}. &\text{by Lemma \ref{lemma_schinzel}}
    \end{align*}
Because of the symmetry of degrees of siblings in $I_k$ shown in Corollary~\ref{cor_naf_deg_symmetry}, $B_{a_{k-2}+v}$ and $B_{2^{k-2}-(a_{k-2}+v)}$ have the same degree. The theorem follows.
\end{proof}
\begin{proposition}\label{prop_degree_lc_C}
     Let $k \ge 3$, and let $0 \le v < \lvert I_{k-2}\rvert$. Then 
     \begin{align*}
        deg(B_{c_k+v}) &= deg(B_{2^{k-2}-(a_{k-2}+v)})+1 \text{, and } \\
        \ell c(B_{c_k+v}) &= \ell c(B_{2^{k-2}-(a_{k-2}+v)})+\ell c(B_{a_{k-2}+v}).
     \end{align*}
\end{proposition}
\begin{proof}
Setting $a=2^{k-2},\ \ m=2,\text{ and } r=a_{k-2}+v$, we apply Lemma \ref{lemma_schinzel}:
    \begin{align*}
        B_{c_k+v} 
        &= B_{(2^{k-1}+a_{k-2})+v}  &\text{by Lemma \ref{lemma_values_ac}}\\
        &= t\cdot B_{2^{k-2}-(a_{k-2}+v)}+(t+1)\cdot B_{a_{k-2}+v}. &\text{by Lemma \ref{lemma_schinzel}}
    \end{align*}
Because of the symmetry of degrees of siblings in $I_k$ shown in Corollary~\ref{cor_naf_deg_symmetry}, $B_{a_{k-2}+v}$ and $B_{2^{k-2}-(a_{k-2}+v)}$ have the same degree. The theorem follows.
\end{proof}
\subsection{Degrees and leading coefficents of the Stern polynomial in $\mathcal{B}_k$}
The Stern polynomials of the integers in $\mathcal{B}_k$ may be expressed as a combination of Stern polynomials of integers in $I_{k-1}$ and integers in $I_d$ with $d<{k-1}$. 
This follows by application of Schinzel's Lemma~\ref{lemma_schinzel} and the lemma comparing degrees of $B_{2^a-r}$ and $B_r$.

Lemma~\ref{lemma_translation_uw} is a translation lemma that allows us to 
express elements in $\mathcal{B}_k$ in the form of $2^{k-1} \pm w$, thus permitting easy application of Lemma~\ref{lemma_schinzel}. \begin{lemma}\label{lemma_translation_uw}
    Let $0 \le u \le 2^{k-1}-b_k$, and let $w=a_{k-2}-1-u$. Then $w \in I_m$, with $m<{k-2}$, and
    \begin{align*}
        a_{k-1}+u &= 2^{k-2}-w, \\
        b_k+u &= 2^{k-1}-w, \text{ and}\\
        2^k-(b_k+u) &= 2^{k-1}+w.\\
    \end{align*}
\end{lemma}
\begin{proof}
We have that $w \in I_m$, for $m<k-2$, since $w < a_{k-2}$.

By the definition of siblings in $I_k$, $a_{k-2}$ and $a_{k-1}-1$ are siblings, so
\begin{align}
    a_{k-1}-1 &= 2^{k-2}-a_{k-2}. &\text{by Lemma~\ref{lemma_naf_symmetricity}} \label{eq_ak}  
    \end{align}
    Likewise, $b_k$ and $c_k-1$ are siblings, so
    \begin{align}
    \nonumber b_k &= 2^k-(c_k-1) &\text{by Lemma~\ref{lemma_naf_symmetricity}}\\
    \nonumber &= 2^k-(2^{k-1}+a_{k-2}-1) &\text{by Lemma~\ref{lemma_values_ac}}\\
    &= 2^{k-1}-a_{k-2}+1. \label{eq_bk} 
\end{align} 
The results follow from these equations and the definition of $w$.
    \begin{align*}
        a_{k-1}+u &= a_{k-1}+a_{k-2}-1-w&\text{by the definition of $w$}\\
        &= (2^{k-2}-a_{k-2}+1)+a_{k-2}-1-w &\text{by Equation (\ref{eq_ak})}\\
        &= 2^{k-2}-w.  \\[.5em]
        b_k+u &= b_k+a_{k-2}-1-w &\text{by the definition of $w$}\\
        &= 2^{k-1}-w. &\text{by Equation (\ref{eq_bk})}\\[.5em]
        2^k-(b_k+u) &= 2^k-b_k-u\\
        &= 2^k-b_k-a_{k-2}+1+w &\text{by the definition of $w$} \\
        &= 2^k-(2^{k-1}+1)+1+w  &\text{by Equation (\ref{eq_bk})}\\
        &= 2^{k-1}+w.
    \end{align*}
\end{proof}
\begin{proposition}\label{prop_deg_B}
    Let $k\ge 3$ and let $0 \le u \le 2^{k-1}-b_k$. Then
    \begin{align*}
        deg(B_{b_k+u})&=deg(B_{a_{k-1}+u})+1, \text{ and}\\
        deg(B_{2^{k}-(b_k+u)})&=deg(B_{a_{k-1}+u})+1.
    \end{align*}
\end{proposition}
\begin{proof}
Let $w$ be as in Lemma~\ref{lemma_translation_uw}. By applying Lemma~\ref{lemma_schinzel}, with $a=2^{k-2}$, $m=2$, and $r=w$,
\begin{align*}
    deg(b_k+u) &= deg(B_{2^{k-1}-w}) &\text{by Lemma~\ref{lemma_translation_uw}} \\
    &= deg(t\cdot B_{2^{k-2}-w} +B_w) &\text{by Lemma~\ref{lemma_schinzel}}\\
    &= deg(B_{2^{k-2}-w})+1. &\text{by Lemma~\ref{lemma_compare_degs}}
\end{align*}
$2^{k}-(b_k+u)$ is the sibling of $b_k+u$, so by Corollary~\ref{cor_naf_deg_symmetry}, 
$$
    deg(B_{2^{k}-(b_k+u)})=deg(B_{a_{k-1}+u})+1.
$$
\end{proof}
\begin{proposition}\label{prop_lc_B_new}
    Let $k\ge 3$ and let $0 \le u \le 2^{k-1}-b_k$. Then
    \begin{align*}
        \ell c(B_{b_k+u})&=\ell c(B_{a_{k-1}+u}), \text{ and}\\
        \ell c(B_{2^{k}-(b_k+u)})&=\ell c(B_{a_{k-1}+u}).
    \end{align*}
\end{proposition}
\begin{proof}
Let $w$ be as in Lemma~\ref{lemma_translation_uw}. By applying Lemma~\ref{lemma_schinzel}, with $a=2^{k-2}$, $m=2$, and $r=w$:
    \begin{align*}
    B_{b_k+u} &= B_{2^{k-1}-w} &\text{by Lemma~\ref{lemma_translation_uw}}\\ 
    &= t\cdot B_{2^{k-2}-w} +B_w. &\text{by Lemma~\ref{lemma_schinzel}}\\
    deg(t \cdot B_{2^{k-2}-w}) &> deg(B_w) &\text{by Lemma~\ref{lemma_compare_degs}}
\end{align*}
So by the above and Lemma~\ref{lemma_translation_uw},  $$ \ell c(B_{b_k+u})=\ell c(B_{2^{k-2} - w}) =\ell c(B_{a_{k-1}+u}).$$ 
Again applying Lemma~\ref{lemma_schinzel}, with $a=2^{k-2}$, $m=2$, and $r=w$:
    \begin{align*}
    B_{2^k-(b_k+u)} &= B_{2^{k-1} + w} &\text{by Lemma~\ref{lemma_translation_uw}}\\
        &= t \cdot B_{2^{k-2} - w} + (t+1)\cdot B_w. &\text{by Lemma~\ref{lemma_schinzel}}\\
    deg(t \cdot B_{2^{k-2}-w}) &> deg((t+1) \cdot B_w). &\text{by Lemma~\ref{lemma_compare_degs}}
    \end{align*}
So  by the above and Lemma~\ref{lemma_translation_uw},  
$$
    \ell c(B_{2^k-(b_k+u)}) 
    = \ell c(B_{2^{k-2} - w})
    = \ell c(B_{a_{k-1}+u}).
$$
\end{proof}
\subsection{Recursions on leading coefficients and degrees of Stern polynomials}
In this section, the recursions on the degrees and on the leading coefficients of $I_k$ are brought together. 

\begin{theorem}\label{theorem_stern_degs_rearranged}
    Let $k \ge 3$, and let $n \in I_k$. Let $0 \le v < \lvert I_{k-2}\rvert$, and let $0 \le u \le 2^{k-1}-b_k$. Then
    \begin{align*}
        deg(B_{a_k+v})&=deg(B_{a_{k-2}+v})+1,\\[.5em]
        deg(B_{b_k+u})&=
            deg(B_{a_{k-1}+u})+1, \\
        deg(B_{2^k-(b_k+u)})&=
            deg(B_{a_{k-1}+u})+1, \\[.5em]
        deg(B_{c_k+v})&=deg(B_{2^{k-2}-(a_{k-2}+v)})+1. 
    \end{align*}
\end{theorem}
\begin{proof}
    The results on the integers in $\mathcal{A}_k$ and $\mathcal{C}_k$ follow from Propositions~\ref{prop_degree_lc_A} and \ref{prop_degree_lc_C}. The result on the integers in $\mathcal{B}_k$ follows from Proposition~\ref{prop_deg_B}.
\end{proof}
\begin{theorem}\label{theorem_stern_lcs_rearranged}
    Let $k \ge 3$, and let $n \in I_k$. Let $0 \le v < \lvert I_{k-2}\rvert$, and let $0 \le u \le 2^{k-1}-b_k$. Then
    \begin{align*}
        \ell c(B_{a_k+v})&=\ell c(B_{a_{k-2}+v}),\\[.5em]
        \ell c(B_{b_k+u})
        &=\ell c(B_{a_{k-1}+u}),\\
        \ell c(B_{2^{k}-(b_k+u)})&=\ell c(B_{a_{k-1}+u}),\\[.5em]
        \ell c(B_{c_k+v})&=\ell c(B_{2^{k-2}-(a_{k-2}+v)}) + \ell c(B_{a_{k-2}+v}).
    \end{align*}
\end{theorem}
\begin{proof}
    The results on the integers in $\mathcal{A}_k$ and $\mathcal{C}_k$ follow from Propositions~\ref{prop_degree_lc_A} and \ref{prop_degree_lc_C}. The result on the integers in $\mathcal{B}_k$ follows from Proposition~\ref{prop_lc_B_new}.
\end{proof}

In \cite{Schinzel2016}, Schinzel proved a theorem expressing the leading coefficients of the Stern polynomial of a single integer $n$ in terms of its binary representation.
The recursion above, however, may be applied \emph{en masse} to the Stern polynomials of all the integers in $I_k$, facilitating comparison between them. 

\section{Optimal BSD representations of integers in $I_k$}\label{sec_weights_BSDs}
We can now apply the weight-distribution theorem to the recursions on Stern polynomials, and derive the number of optimal BSD representations of an integer $n$, along with their Hamming weight.

\begin{definition}[Optimal BSD representations]
    A BSD representation of $n$ is called \emph{optimal} if it has same Hamming weight as the reduced NAF representation of $n$. We refer to the number of optimal representations of $n$ as $M(n)$, and to the number of $0$s in an optimal representation as $Z(n)$.
\end{definition} 
\begin{theorem}\label{theorem_minimal_weight}
    Let $k \ge 3$, and let $n \in I_k$. Let $0 \le v < \lvert I_{k-2}\rvert$, and let $0 \le u \le 2^{k-1}-b_k$. 
    Then
    \begin{align*}
        Z(a_k+v)&=Z(a_{k-2}+v)+1,\\[.5em]
        Z(b_k+u)&=Z(a_{k-1}+u)+1 ,\\
        Z(2^k-(b_k+u))&= Z(a_{k-1}+u)+1, \\[.5em]
        Z(c_k+v)&=Z(2^{k-2}-(a_{k-2}+v))+1.
    \end{align*}
\end{theorem}
\begin{proof}
    By Theorem~\ref{theorem_weightdist}, the number of $0$s in the NAF-representation of $n$ is the same as $deg(B_{2^k-n})$.
    By Lemma~\ref{lemma_compare_degs}, $deg(B_{2^k-n})=deg(B_n)$. The theorem follows by application of these results to Theorem~\ref{theorem_stern_degs_rearranged}, the result on the degree of a Stern polynomial.
\end{proof}
\begin{theorem}\label{theorem_optimal_BSDs}
    Let $k \ge 3$, and let $n \in I_k$. Let $0 \le v < \lvert I_{k-2}\rvert$, and let $0 \le u \le 2^{k-1}-b_k$. 
    Then
    \begin{align*}
        &M(a_k+v)=M(2^{k-2}-(a_{k-2}+v)) + M(a_{k-2}+v),\\[.5em]
        &M(b_k+u)=
           M(2^{k-1}-(a_{k-1}+u)), \\
        &M(2^k-(b_k+u))=M(2^{k-1}-(a_{k-1}+u)),  \\[.5em]
        &M(c_k+v)=M(a_{k-2}+v).
    \end{align*}
\end{theorem}
\begin{proof}
    For the $M(a_k+v)$ and $M(c_k+v)$ cases, let $x$ be as in Lemma~\ref{lemma_sibs_across_k}.
   \begin{align*}
         M(a_k+v)
        &= \ell c(B_{2^k-(a_k+v)}) &\text{by Theorem~\ref{theorem_weightdist}} \\
        &= \ell c(B_{c_k+x})&\text{by Lemma~\ref{lemma_sibs_across_k}}\\
        &= \ell c(B_{a_{k-2}+x})+\ell c(B_{2^{k-2}-(a_{k-2}+x)}) &\text{by Theorem~\ref{theorem_stern_lcs_rearranged}}\\
        &= \ell c(B_{2^{k-2}-(a_{k-2}+v)}) + \ell c(B_{a_{k-2}+v})  &\text{by definition of $x$}\\
        &= M(a_{k-2}+v) + M(2^{k-2}-(a_{k-2}+v)).  &\text{by Theorem~\ref{theorem_weightdist}}\\[1em]
        M(c_k+v)
        &= \ell c(B_{2^k-(c_k+v)}) &\text{by Theorem~\ref{theorem_weightdist}} \\
        &= \ell c(B_{a_k+x}) &\text{by Lemma~\ref{lemma_sibs_across_k}} \\
        &= \ell c(B_{a_{k-2}+x}) &\text{by Theorem~\ref{theorem_stern_lcs_rearranged}}\\
        &= \ell c(B_{2^{k-2}-(a_{k-2}+v)})   &\text{by the the definition of $x$}\\
        &= M(a_{k-2}+v). &\text{by Theorem~\ref{theorem_weightdist}}
    \end{align*}
    The $M(b_k+u)$ and $M(2^k-(b_k+u))$ cases follow from Theorems~\ref{theorem_weightdist} and \ref{theorem_stern_lcs_rearranged}.
    \begin{align*}
        M(b_k+u)
        &= \ell c(B_{2^k-(b_k+u)})  &\text{by Theorem~\ref{theorem_weightdist}}\\
        &= \ell c(B_{a_{k-1}+u})  &\text{by Theorem~\ref{theorem_stern_lcs_rearranged}}\\
        &= M(2^{k-1}-(a_{k-1}+u)).  &\text{by Theorem~\ref{theorem_weightdist}}\\[.5em] 
        M(2^k-(b_k+u))
        &= \ell c(B_{2^k-(2^k-(b_k+u))})   &\text{by Theorem~\ref{theorem_weightdist}} \\
        &= \ell c(B_{b_k+u})  \\
        &= \ell c(B_{a_{k-1}+u})  &\text{by Theorem~\ref{theorem_stern_lcs_rearranged}}\\
        &= M(2^{k-1}-(a_{k-1}+u)). &\text{by Theorem~\ref{theorem_weightdist}}
    \end{align*}
 \end{proof}
Fig.~\ref{fig:charts} shows three examples of distributions of the number of optimal representations of integers in $I_k$, for $k=14, 15$ and $16$. Fig.~\ref{fig:chart_0} shows the number of $0$s in the optimal representations of integers in $I_{16}$. Here in the text, we will discuss Fig.~\ref{fig:charts}, and provide Fig.~\ref{fig:chart_0} for comparison.
 \begin{figure}[H]
 \hfill
    \begin{subfigure}[t]{.51\textwidth}
      \includegraphics[width=\textwidth]{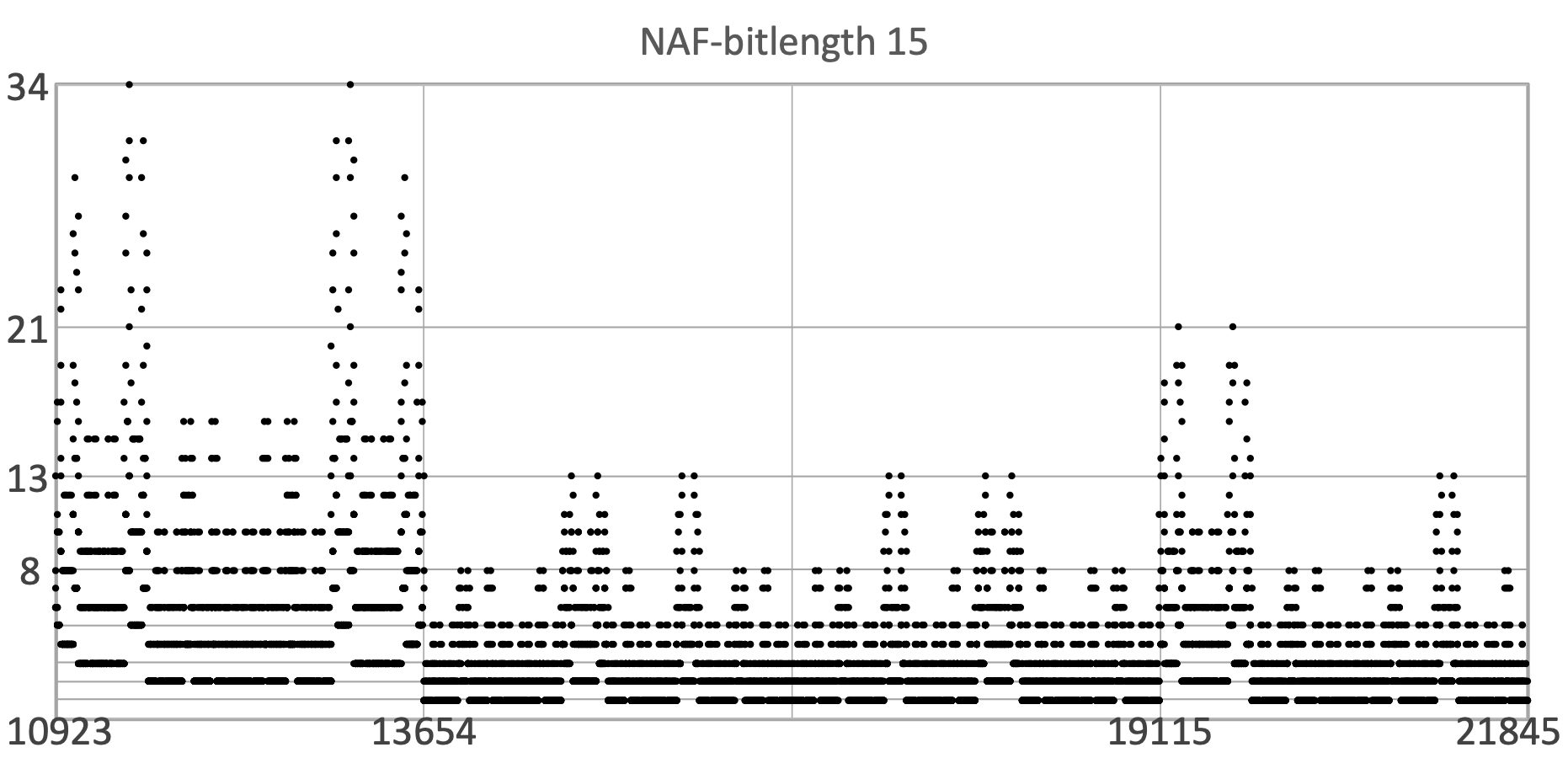}
      \caption{The number of optimal representations of integers in $I_{15}$.}
      \label{fig:charts_15}
    \end{subfigure}
    \begin{subfigure}[t]{.27\textwidth}
      \includegraphics[width=\textwidth]{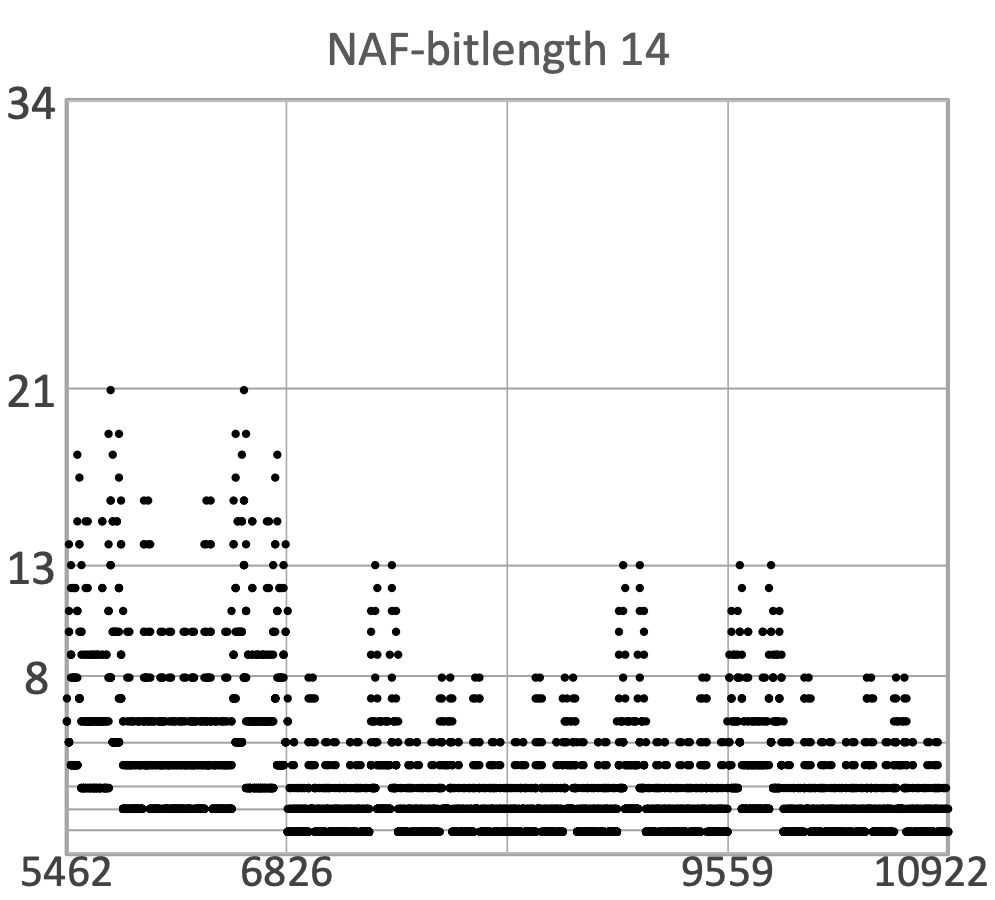}
      \caption{The number of optimal representations of integers in $I_{14}$.}
      \label{fig:charts_14}
    \end{subfigure}

    \begin{subfigure}[t]{1\textwidth}
      \includegraphics[width=\textwidth]{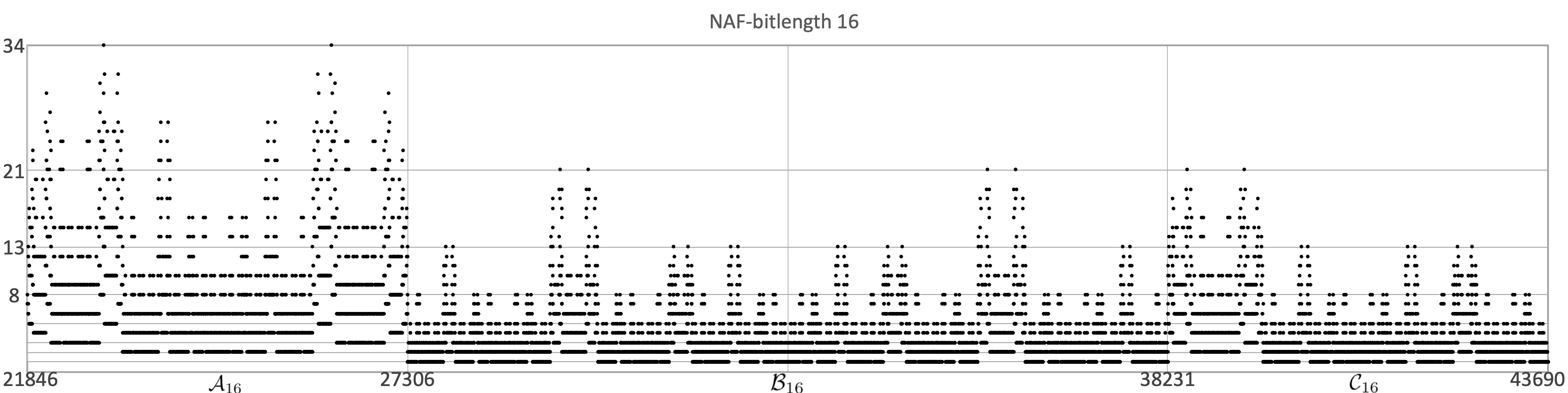}
      \caption{The number of optimal representations of integers in $I_{16}$. The distribution for $I_{14}$ from (\subref{fig:charts_14}) can be seen to be exactly that of $\mathcal{C}_{16}$ here, and the distribution in the second half of $I_{15}$  from (\subref{fig:charts_15}) can be seen to be exactly that of the second half of $\mathcal{B}_{16}$ here.}
      \label{fig:charts_16}
    \end{subfigure}
    \caption{The number of optimal representations of integers in $I_{14}$, $I_{15}$, and $I_{16}$. 
    }
\label{fig:charts}
 \end{figure}
The $x$-axes represent the integers in $I_k$. These axes are delineated at $a_k$, $b_k$ and $c_k$ and $2^k-1$.  
In Fig.~\ref{fig:charts}, the $y$ axes represent the number of representations of a given integer $n \in I_k$.The tick marks on the $y$-axes are the Fibonacci numbers, which are the maxima within a NAF-interval. 
The maxima have value $F_{\ceil{\frac{k}{2}}+1}$ and  are reached in $\mathcal{A}_k$, as discussed in \cite{ganesan2004,grabner06,sawada2007,Tuma_2015}. Relative maxima can also be seen within subintervals, and these are also Fibonacci numbers.

In general, the distributions of the number of optimal representations in $I_k$ for odd $k$ have similar shapes, as do the distributions for even $k$. 
One of the differences between odd and even NAF-bitlengths $k$ is that the relative maxima in $\mathcal{B}_k$ are $F_{\frac{k}{2}}$ for $k$ even, but are only $F_{\frac{k-1}{2}}$ for $k$ odd. 

The symmetries within $\mathcal{A}_k$ and $\mathcal{B}_k$ are evident. The $\mathcal{C}_k$ are not symmetric, but instead have the same values as $\mathcal{A}_{k-2}$, giving a fractal structure in the $\mathcal{C}_i$.

The distributions of the leading coefficients of the Stern polynomials $n \in I_k$ are easily obtained from the distributions of the number of optimal representations: by Theorem~\ref{theorem_weightdist}, the Stern polynomial distributions are those of the number of optimal representations, but reflected about the $I_k$-midpoint $x=2^{k-1}$.
 \begin{figure}[H]
 \centering
      \includegraphics[width=.98\textwidth]{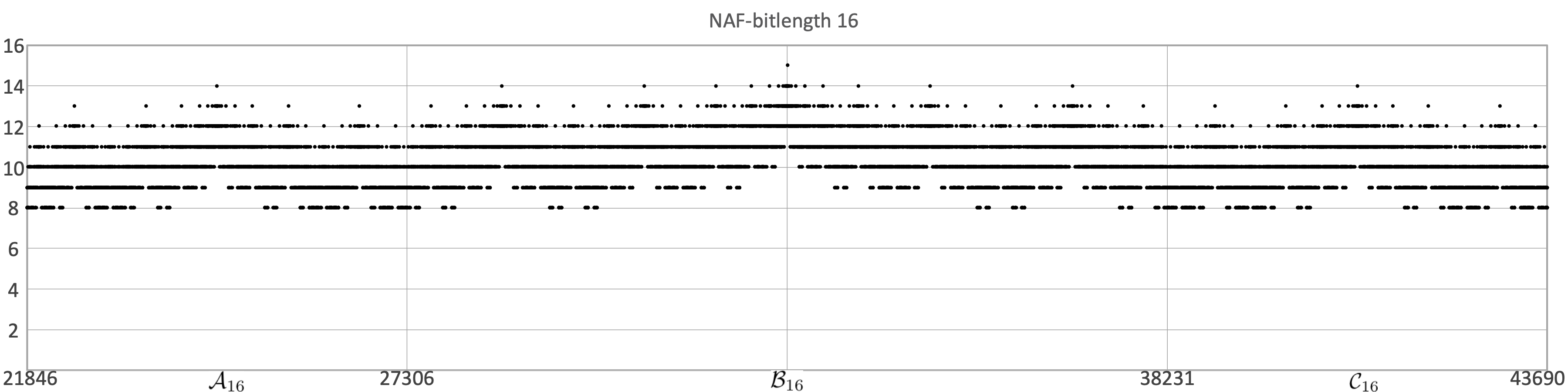}
      \caption{The number of $0$s in the optimal representations of integers in $I_{16}$.}
      \label{fig:chart_0_16}
\label{fig:chart_0}
 \end{figure}

\section{Algorithms}
We present in this section two algorithms counting the number of $k$-bit optimal BSD representations of integers in $I_k$, and giving the number of $0$s in such representations. 
They calculate for all the integers in $I_k$ recursively, depending on previously calculated values for the integers in $I_{k-1}$ and $I_{k-2}$.

They are quite simple. Algorithm~\ref{alg_stern_degs} does a copy and an increment for each $n$.   Algorithm~\ref{alg_stern_Ns} does only a single copy for three-quarters of the integers in $I_k$ and a two-term sum for the rest of $I_k$. Algorithm~\ref{alg_stern_Ns} is illustrated in Fig~\ref{fig:alg}. 

These algorithms are $\mathcal{O}(n)$, with $n \in I_k$. They are $\mathcal{O}(1)$ on average per integer, but must calculate for all $m<n$.
They are embarrassingly parallel within a NAF-interval $I_k$; however, the $i$-loop that calculates each NAF-interval must be performed sequentially, because of the NAF-interval dependency.

The advantage of these algorithms is that they permit comparison between all integers in $I_k$, since all values for $I_k$ are calculated in one loop iteration. It is often advantageous to work with integers of a given NAF-bitlength having a relatively large number of optimal representations, or else having optimal representations of relatively small weight. These algorithms allow one to choose between many integers having lengths, weights and optimal representations suited to the needs at hand.

In \cite{grabner06}, Grabner and Heuberger gave relations that count the number of optimal representations of an integer $n$ using transducers. In \cite{Schinzel2016}, Schinzel gave an method for calculating the leading coefficients of the Stern polynomial of an integer $n$, in terms of its binary decomposition. From this, the number of optimal representations of $n$ may be calculated from application of the weight-distribution Theorem~\ref{theorem_weightdist}.

Both of these are  $\mathcal{O}(\log(n))$ and would be algorithms of choice if one wanted the result for only one $n$. 
However, calculating for all $n \in I_k$, both would be  $\mathcal{O}(n\log(n))$. So if comparison is desired, the $\mathcal{O}(n)$ Algorithm~\ref{alg_stern_Ns} presented below is preferable.
\begin{algorithm}[H]\caption{Number of $0$s in optimal BSD representations}\label{alg_stern_degs}
\setstretch{1.25}
	\begin{algorithmic}[1] 
		\Function{Zeros-Opt-BSD}{$k, Z$} \funclabel{zeros-opt-bsd:k,Z}
		\State $Z[0] \gets 0$, $Z[1] \gets 0$, $Z[2] \gets 1$, $a_{i-2} \gets 0$, $a_{i-1} \gets 1$, $a_i \gets 2$
		\For{$i \gets 3, k$} 
		    \State $a_{i-2} \gets a_{i-1}$, $a_{i-1} \gets a_i$, $a_i \gets 2^{i-2}+a_{i-2}$
		    \Comment By Lemma~\ref{lemma_values_ac}
		    \State $loop\_length \gets 2^{i-1}-a_i$
		    \State $Z[2^{i-1}] \gets i-1$
		    \Comment By Theorem~\ref{theorem_minimal_weight} 
		    \For{$j \gets 0, loop\_length-1$} 
		    \Comment By Theorem~\ref{theorem_minimal_weight} 
		        \State $Z[a_i+j] \gets Z[a_{i-2}+j]+1$
		        \State $Z[2^i-(a_i+j)] \gets Z[a_{i-2}+j]+1$
		    \EndFor
        \EndFor
    \EndFunction
	\end{algorithmic}
\end{algorithm}
\begin{algorithm}[H]\caption{Number of optimal BSD representations}\label{alg_stern_Ns}
\setstretch{1.25}
	\begin{algorithmic}[1] 
		\Function{Num-Opt}{$k, M$} \funclabel{Num-Opts2:k,M}
		\State $M[0] \gets 0$,  $M[1] \gets 1$, $M[2] \gets 1$, $a_{0} \gets 0$, $a_{1} \gets 1$, $a_2 \gets 2$
		\For{$i \gets 3, k$} 
		    \State $a_{i-2} \gets a_{i-1}$, $a_{i-1} \gets a_i$, $a_i \gets 2^{i-2}+a_{i-2}$
		    \Comment By Lemma~\ref{lemma_values_ac}
		    \If {$i$ is even}
		        \State $ac\_loop\_length \gets a_{i-2}-1$
		    \Comment By Lemma~\ref{lemma_length_naf}
		        \State $b\_loop\_length \gets a_{i-2}$
		    \Comment By Lemma~\ref{lemma_values_bc}
		    \Else
		        \State $ac\_loop\_length \gets a_{i-2}$
		    \Comment By Lemma~\ref{lemma_length_naf}
		        \State $b\_loop\_length \gets a_{i-2}-1$
		    \Comment By Lemma~\ref{lemma_values_bc}
		    \EndIf
		    \State $b_i \gets a_{i}+ac\_loop\_length$
		    \Comment By $I_k$ partition
		    \State $c_i \gets 2^{i-1}+a_{i-2}$
		    \Comment By Lemma~\ref{lemma_values_ac}
		    \For{$h \gets 0, ac\_loop\_length-1$}
		    \Comment By Theorem~\ref{theorem_optimal_BSDs} 
		        \State $M[a_i+h] \gets M[a_{i-2}+h] + M[2^{i-2}-(a_{i-2}+h)]$
		        \State $M[c_i+h] \gets M[a_{i-2}+h]$
		    \EndFor
		    \State $M[2^{i-1}] \gets 1$
		    \Comment By Theorem~\ref{theorem_optimal_BSDs} 
		    \For{$j \gets 0, b\_loop\_length-1$} 
		    \Comment By Theorem~\ref{theorem_optimal_BSDs} 
		        \State $M[b_i+j] \gets M[2^{i-1}-(a_{i-1}+j)]$
		        \State $M[2^{i}-(b_i+j)] \gets M[2^{i-1}-(a_{i-1}+j)]$
		    \EndFor
        \EndFor
    \EndFunction
	\end{algorithmic}
\end{algorithm}
\newpage
 \begin{figure}[H]
   \centering
    \begin{subfigure}[t]{1\textwidth}
   \centering
       \includegraphics[width=.8\textwidth]{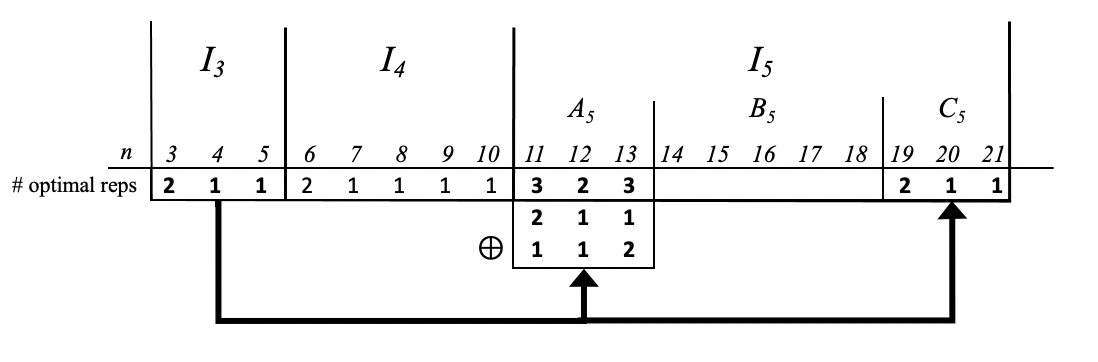}
      \caption{The number of optimal representations of integers in $\mathcal{A}_5$ and $\mathcal{C}_5$ are calculated in the $h$ loop. In general, the entries from $I_{k-2}$ in order are added to the entries of $I_{k-2}$ in reverse order, and the result is entered into $\mathcal{A}_k$. The entries of $I_{k-2}$ are entered into $\mathcal{C}_k$ in order.}
    \end{subfigure}
    \hfill
    \begin{subfigure}[t]{1\textwidth}
   \centering
      \includegraphics[width=.8\textwidth]{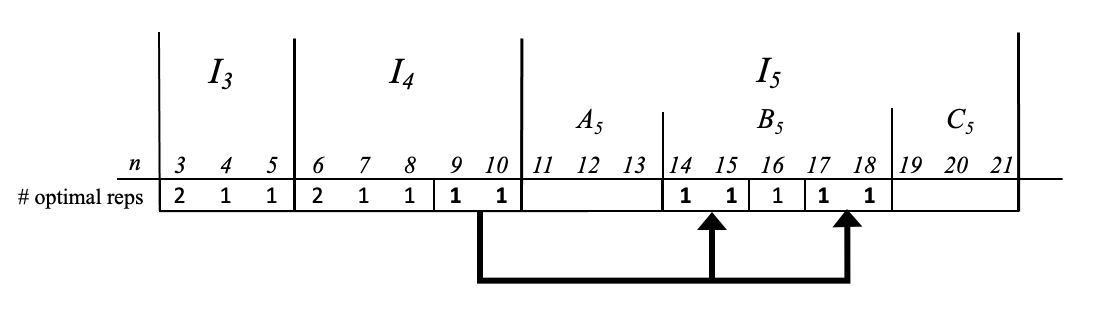}
      \caption{The number of optimal representations of integers in $\mathcal{B}_5$ are calculated in the $j$ loop. In general, the entries from  $I_{k-1}$ are entered into the first half of $\mathcal{B}_k$ in reverse order, and into the last half of $\mathcal{B}_k$ in order. (The reverse ordering is not evident in this small example, because the relevant entries of $I_4$ are palindromic.) The entry for $2^4$ is $1$.}
    \end{subfigure}
    \caption{An illustration of the $h$ and $j$ loops inside of the $i$ loop in Algorithm~\ref{alg_stern_Ns}, showing the calculation of the number of optimal BSD representations of $n \in I_5$. The calculation inside of the $i$ loop is embarassingly parallel. Every calculation is a copy from $I_3$, or the sum of two entries from $I_3$, or a copy from $I_4$.}
\label{fig:alg}
 \end{figure}

\section{Acknowledgements}
	The author wishes to thank the anonymous reviewers of a previous paper for helpful reference suggestions, and Andrew Alexander, Matthew Broussard, Vanessa Job and Nathan Kodama for much discussion. 

\bibliographystyle{spmpsci}
\bibliography{stern-diatomic}

\begin{thebibliography}{10}
\providecommand{\url}[1]{{#1}}
\providecommand{\urlprefix}{URL }
\expandafter\ifx\csname urlstyle\endcsname\relax
  \providecommand{\doi}[1]{DOI~\discretionary{}{}{}#1}\else
  \providecommand{\doi}{DOI~\discretionary{}{}{}\begingroup
  \urlstyle{rm}\Url}\fi

\bibitem{avizienis1961signed}
Avizienis, A.: Signed-digit number representations for fast parallel
  arithmetic.
\newblock IRE Trans. Comput. \textbf{EC-10}(3), 389--400 (1961)

\bibitem{booth1951}
Booth, A.D.: {A Signed Binary Multiplication Technique}.
\newblock Quart. J. Mech. Appl. Math. \textbf{4}(2), 236--240 (1951).
\newblock \doi{10.1093/qjmam/4.2.236}.
\newblock \urlprefix\url{https://doi.org/10.1093/qjmam/4.2.236}

\bibitem{dilcher10}
Dilcher, K., Tomkins, H.: Square classes and divisibility properties of {S}tern
  polynomials.
\newblock Integers \textbf{18} (2018)

\bibitem{egecioglu90}
E\u{g}ecio\u{g}lu, {\"{O}}., Ko{\c{c}}, {\c{C}}.K.: Fast modular
  exponentiation.
\newblock Proceedings of 1990 Bilkent International Conference on New Trends in
  Communication, Control, and Signal Processing \textbf{1}, 188--194 (1990)

\bibitem{ganesan2004}
Ganesan, P., Manku, G.S.: Optimal routing in {C}hord.
\newblock In: ACM SIAM Symposium on Discrete Algorithms (SODA 2004) (2004).
\newblock \urlprefix\url{http://ilpubs.stanford.edu:8090/640/}

\bibitem{grabner06}
Grabner, P., Heuberger, C.: On the number of optimal base 2 representations of
  integers.
\newblock Des. Codes Cryptogr. \textbf{40}, 25--39 (2006).
\newblock \doi{10.1007/s10623-005-6158-y}

\bibitem{KLAVZAR200786}
Klavžar, S., Milutinović, U., Petr, C.: Stern polynomials.
\newblock Adv. in Appl. Math. \textbf{39}(1), 86 -- 95 (2007).
\newblock \doi{https://doi.org/10.1016/j.aam.2006.01.003}.
\newblock
  \urlprefix\url{http://www.sciencedirect.com/science/article/pii/S0196885806000807}

\bibitem{koblitz91}
Koblitz, N.: C{M}-curves with good cryptographic properties.
\newblock In: Advances in Cryptology---{CRYPTO} '91 ({S}anta {B}arbara, {CA},
  1991), \emph{Lecture Notes in Comput. Sci.}, vol. 576, pp. 279--287.
  Springer, Berlin (1992).
\newblock \doi{10.1007/3-540-46766-1_22}.
\newblock \urlprefix\url{https://doi.org/10.1007/3-540-46766-1_22}

\bibitem{monroe21}
Monroe, L.: Binary signed-digit integers and the {S}tern diatomic sequence.
\newblock \url{https://arxiv.org/abs/2108.11495} (to appear in \textit{Designs,
  Codes and Cryptography})

\bibitem{morain90}
Morain, F., Olivos, J.: Speeding up the computations on an elliptic curve using
  addition-subtraction chains.
\newblock RAIRO Theor. Inform. Appl. \textbf{24}, 531--544 (1990).
\newblock \doi{10.1051/ita/1990240605311}

\bibitem{REITWIESNER1960231}
Reitwiesner, G.W.: Binary Arithmetic, \emph{Advances in Computers}, vol.~1, pp.
  231 -- 308.
\newblock Elsevier (1960).
\newblock \doi{https://doi.org/10.1016/S0065-2458(08)60610-5}.
\newblock
  \urlprefix\url{http://www.sciencedirect.com/science/article/pii/S0065245808606105}

\bibitem{reznick90}
Reznick, B.: Some binary partition functions.
\newblock In: Analytic Number Theory, \emph{Progress in Mathematics}, vol.~85,
  pp. 451--477. Birkhauser Boston (1990).
\newblock \doi{10.1007/978-1-4612-3464-7_29}.
\newblock \urlprefix\url{http://dx.doi.org/10.1007/978-1-4612-3464-7_29}

\bibitem{sawada2007}
Sawada, J.: A simple {G}ray code to list all minimal signed binary
  representations.
\newblock SIAM J. Discrete Math. \textbf{21}(1), 16--25 (2007).
\newblock \doi{10.1137/050641405}.
\newblock \urlprefix\url{https://doi.org/10.1137/050641405}

\bibitem{Schinzel2011}
Schinzel, A.: On the factors of {S}tern polynomials: Remarks on the preceding
  paper of {M}. {U}las.
\newblock Publicationes Mathematicae \textbf{79} (2011).
\newblock \doi{10.5486/PMD.2011.5110}

\bibitem{Schinzel2016}
Schinzel, A.: The leading coefficients of {S}tern polynomials.
\newblock In: From Arithmetic to Zeta-Functions: Number Theory in Memory of
  {W}olfgang {S}chwarz, pp. 427--434. Springer International Publishing (2016).
\newblock \doi{10.1007/978-3-319-28203-9_25}.
\newblock \urlprefix\url{https://doi.org/10.1007/978-3-319-28203-9_25}

\bibitem{Shallit92aprimer}
Shallit, J.: A primer on balanced binary representations.
\newblock \url{http://cs.uwaterloo.ca/~shallit/Papers/bbr.pdf} (1992)

\bibitem{shannon50}
Shannon, C.E.: A symmetrical notation for numbers.
\newblock Amer. Math. Monthly \textbf{57}(2), pp. 90--93 (1950).
\newblock \urlprefix\url{http://www.jstor.org/stable/2304993}

\bibitem{Tuma_2015}
T{\r{u}}ma, J., V{\'{a}}bek, J.: On the number of binary signed digit
  representations of a given weight.
\newblock Commentationes Mathematicae Universitatis Carolinae \textbf{56}(3),
  287--306 (2015).
\newblock \doi{10.14712/1213-7243.2015.129}

\bibitem{ulas2012}
Ulas, M.: Arithmetic properties of the sequence of degrees of {S}tern
  polynomials and related results.
\newblock Int. J. Number Theory \textbf{8}, 669--687 (2012).
\newblock \doi{10.1142/S1793042112500388}

\bibitem{ulas2011}
{Ulas}, M., {Ulas}, O.: {On certain arithmetic properties of Stern
  polynomials}.
\newblock Publ. Math. Debrecen \textbf{79}(1-2), 55--81 (2011)

\end{thebibliography}

\end{document}